\documentclass[11pt]{article}
\usepackage{amsfonts,amsmath,amssymb}
\usepackage[left,modulo]{lineno}

\DeclareMathSymbol{\lcurs}{\mathord}{letters}{"0B}
\newtheorem{theorem}{\bf Theorem}[section]
\newtheorem{corollary}[theorem]{\bf Corollary}
\newtheorem{lemma}[theorem]{\bf Lemma}
\newtheorem{proposition}[theorem]{\bf Proposition}

\newcommand{\qed}{\hfill $\square$ \bigskip}

\newenvironment{wlist}
{\vspace{-10pt}
\begin{list}{}
{\setlength{\labelwidth}{15mm} \setlength{\partopsep}{0pt}
\setlength{\parskip}{0pt} \setlength{\topsep}{0pt}
\setlength{\itemsep}{0pt} \setlength{\parsep}{0pt}
\setlength{\labelsep}{10pt} \setlength{\leftmargin}{15mm}}
\item[]} {\end{list} \vspace{2pt}
\smallskip}

\linenumbers
\begin{document}

\title{Maximal hypercubes in Fibonacci and Lucas cubes}

\author{
Michel Mollard \thanks{Partially supported by the ANR Project GraTel (Graphs for Telecommunications), ANR-blan-09-blan-0373-01}\\
CNRS Universit\'e Joseph Fourier\\
Institut Fourier, BP 74 \\
100 rue des Maths, 38402 St Martin d'H\`eres Cedex, France \\
e-mail:  \tt{michel.mollard@ujf-grenoble.fr}
}

\date{}

\maketitle

\begin{abstract}
The Fibonacci cube $\Gamma_n$ is the subgraph of the 
hypercube induced by the binary strings  that contain no
two consecutive 1's. The Lucas cube $\Lambda_n$ is obtained
from $\Gamma_n$ by removing vertices that start 
and end with 1. We characterize maximal induced hypercubes in  $\Gamma_n$ and $\Lambda_n$ and deduce for any $p\leq n$ the number of maximal $p$-dimensional hypercubes in these graphs.
\end{abstract}


\section{Introduction}
An interconnection topology can be represented by a graph $G=(V,E)$, where $V$ denotes the processors and $E$ the communication links.
The \emph{distance} $d_G(u,v)$ between two vertices $u,v$ of a graph $G$ is the length of a shortest path connecting $u$ and $v$.  
An \emph{isometric} subgraph $H$ of a graph $G$ is an induced subgraph such that for any vertices $u,v$ of $H$ we have $d_H(u,v)=d_G(u,v)$.

The \emph{hypercube} of dimension $n$ is the graph $Q_n$ whose vertices are the binary strings of length $n$ where two vertices are adjacent if they differ in exactly one coordinate.
The \emph{weight} of a vertex, $w(u)$, is the number of $1$ in the string $u$. Notice that the graph distance between two vertices of $Q_n$ is equal to the \emph{Hamming distance} of the strings, the number of coordinates they differ.
The hypercube is a popular interconnection network because of its structural properties.\\
\indent Fibonacci cubes and Lucas cubes were introduced in \cite{Hsu1993Fibonacci} and \cite{munarini2001lucas} as new interconnection networks. They are isometric subgraphs of $Q_n$ and have also recurrent structure.
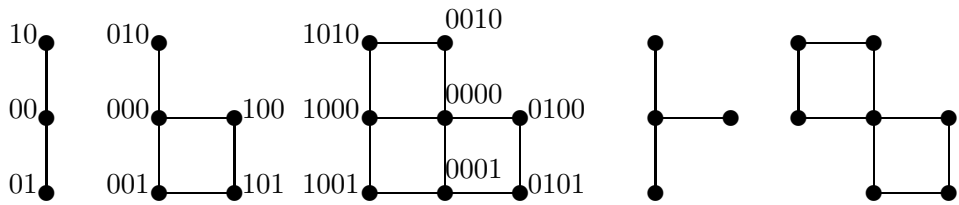
\begin{figure}[!p]
  \centering
\setlength{\unitlength}{1 mm}
\begin{picture}(160, 30)

\newsavebox{\gtwo}
\savebox{\gtwo}
  (30,30)[bl]
  {
  
\put(10,10){\circle*{2}}
\put(10,20){\circle*{2}}
\put(10,30){\circle*{2}}

\put(10,10){\line(0,1){10}}
\put(10,20){\line(0,1){10}}

\put(5,10){$01$}
\put(5,20){$00$}
\put(5,30){$10$}

  }
\newsavebox{\gthree}
\savebox{\gthree}
  (40,30)[bl]
  {
  
\put(10,10){\circle*{2}}
\put(10,20){\circle*{2}}
\put(10,30){\circle*{2}}
\put(20,10){\circle*{2}}
\put(20,20){\circle*{2}}

\put(10,10){\line(1,0){10}}
\put(10,20){\line(1,0){10}}
\put(10,10){\line(0,1){10}}
\put(10,20){\line(0,1){10}}
\put(20,10){\line(0,1){10}}

\put(3,10){$001$}
\put(3,20){$000$}
\put(3,30){$010$}
\put(21,20){$100$}
\put(21,10){$101$}
  }
  
 \newsavebox{\gfour}
\savebox{\gfour}
  (40,30)[bl]
  {
 
\put(10,10){\circle*{2}}
\put(10,20){\circle*{2}}
\put(10,30){\circle*{2}}
\put(20,10){\circle*{2}}
\put(20,20){\circle*{2}}
\put(20,30){\circle*{2}}
\put(30,10){\circle*{2}}
\put(30,20){\circle*{2}}

\put(10,10){\line(1,0){10}}
\put(10,20){\line(1,0){10}}
\put(10,30){\line(1,0){10}}
\put(10,10){\line(0,1){10}}
\put(10,20){\line(0,1){10}}

\put(20,10){\line(1,0){10}}
\put(20,20){\line(1,0){10}}

\put(20,10){\line(0,1){10}}
\put(20,20){\line(0,1){10}}
\put(30,10){\line(0,1){10}}

\put(20,12){$0001$}
\put(20,22){$0000$}
\put(20,32){$0010$}
\put(31,20){$0100$}
\put(1,20){$1000$}
\put(1,10){$1001$}
\put(1,30){$1010$}
\put(31,10){$0101$}
  }
 \newsavebox{\lthree}
\savebox{\lthree}
  (40,30)[bl]
  {
  
\put(10,10){\circle*{2}}
\put(10,20){\circle*{2}}
\put(10,30){\circle*{2}}

\put(20,20){\circle*{2}}

\put(10,20){\line(1,0){10}}
\put(10,10){\line(0,1){10}}
\put(10,20){\line(0,1){10}}

 }
 
\newsavebox{\lfour}
\savebox{\lfour}
  (40,30)[bl]
  {

\put(10,20){\circle*{2}}
\put(10,30){\circle*{2}}
\put(20,10){\circle*{2}}
\put(20,20){\circle*{2}}
\put(20,30){\circle*{2}}
\put(30,10){\circle*{2}}
\put(30,20){\circle*{2}}

\put(10,20){\line(1,0){10}}
\put(10,30){\line(1,0){10}}

\put(10,20){\line(0,1){10}}

\put(20,10){\line(1,0){10}}
\put(20,20){\line(1,0){10}}

\put(20,10){\line(0,1){10}}
\put(20,20){\line(0,1){10}}
\put(30,10){\line(0,1){10}}  }
  
\put(0,-5){\usebox{\gtwo}}  
\put(15,-5){\usebox{\gthree}}
\put(43,-5){\usebox{\gfour}}
\put(81,-5){\usebox{\lthree}}  
\put(100,-5){\usebox{\lfour}}
\end{picture}
\caption{$\Gamma_2=\Lambda_2$, $\Gamma_3$, $\Gamma_4$ and  $\Lambda_3$, $\Lambda_4$}
\end{figure}

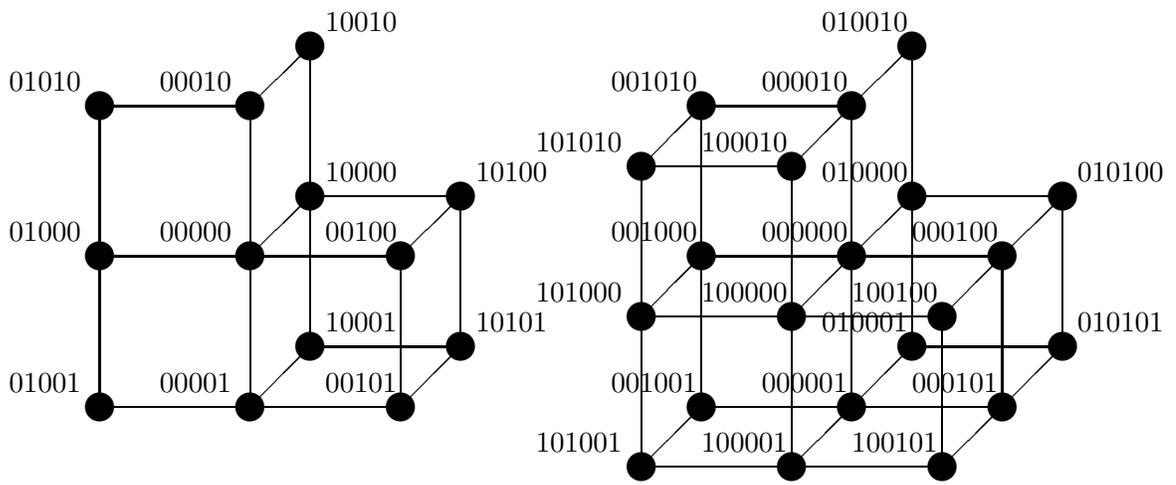
\begin{figure}[!p]
\centering
\setlength{\unitlength}{2mm}
\begin{picture}(80, 35)
 \newsavebox{\ggfour}
\savebox{\ggfour}
  (40,40)[bl]
  {
 
\put(10,10){\circle*{2}}
\put(10,20){\circle*{2}}
\put(10,30){\circle*{2}}
\put(20,10){\circle*{2}}
\put(20,20){\circle*{2}}
\put(20,30){\circle*{2}}
\put(30,10){\circle*{2}}
\put(30,20){\circle*{2}}

\put(10,10){\line(1,0){10}}
\put(20,10){\line(1,0){10}}
\put(10,20){\line(1,0){10}}
\put(20,20){\line(1,0){10}}
\put(10,30){\line(1,0){10}}
\put(10,10){\line(0,1){10}}
\put(10,20){\line(0,1){10}}
\put(20,10){\line(0,1){10}}
\put(20,20){\line(0,1){10}}
\put(30,10){\line(0,1){10}}
  }
\newsavebox{\gfive}
\savebox{\gfive}
  (40,40)[bl]
  {
  
\put(0,0){\usebox{\ggfour}}

\put(24,14){\circle*{2}}
\put(24,24){\circle*{2}}
\put(24,34){\circle*{2}}
\put(34,14){\circle*{2}}
\put(34,24){\circle*{2}}
\put(24,14){\line(1,0){10}}
\put(24,24){\line(1,0){10}}
\put(24,14){\line(0,1){10}}
\put(24,24){\line(0,1){10}}
\put(34,14){\line(0,1){10}}

\put(20,10){\line(1,1){4}}
\put(20,20){\line(1,1){4}}
\put(20,30){\line(1,1){4}}
\put(30,10){\line(1,1){4}}
\put(30,20){\line(1,1){4}}
  }
 \newsavebox{\gsix}
\savebox{\gsix}
  (40,40)[bl]
  {
 \put(0,0){\usebox{\gfive}}

\put(6,6){\circle*{2}}
\put(6,16){\circle*{2}}
\put(6,26){\circle*{2}}
\put(16,6){\circle*{2}}
\put(16,16){\circle*{2}}
\put(16,26){\circle*{2}}
\put(26,6){\circle*{2}}
\put(26,16){\circle*{2}}
\put(6,6){\line(1,0){10}}
\put(16,6){\line(1,0){10}}
\put(6,16){\line(1,0){10}}
\put(16,16){\line(1,0){10}}
\put(6,26){\line(1,0){10}}
\put(6,6){\line(0,1){10}}
\put(6,16){\line(0,1){10}}
\put(16,6){\line(0,1){10}}
\put(16,16){\line(0,1){10}}
\put(26,6){\line(0,1){10}}

\put(6,6){\line(1,1){4}}
\put(6,16){\line(1,1){4}}
\put(6,26){\line(1,1){4}}
\put(16,6){\line(1,1){4}}
\put(16,16){\line(1,1){4}}
\put(16,26){\line(1,1){4}}
\put(26,6){\line(1,1){4}}
\put(26,16){\line(1,1){4}}

\put(14,11){$000001$}
\put(14,21){$000000$}
\put(14,31){$000010$}
\put(24,21){$000100$}
\put(4,21) {$001000$}
\put(18,25){$010000$}
\put(10,17){$100000$}
\put(4,11) {$001001$}
\put(4,31) {$001010$}
\put(-1,7){$101001$}
\put(-1,17){$101000$}
\put(-1,27){$101010$}
\put(10,7){$100001$}
\put(10,27){$100010$}
\put(18,35){$010010$}
\put(24,11){$000101$}
\put(35,25){$010100$}
\put(35,15){$010101$}
\put(18,15){$010001$}
\put(20,17){$100100$}
\put(20,7){$100101$}

  }
  
\put(0,-5){\usebox{\gfive}}

\put(14,6){$00001$}
\put(14,16){$00000$}
\put(14,26){$00010$}
\put(25,16){$00100$}
\put(4,16){$01000$}
\put(25,20){$10000$}
\put(4,26){$01010$}
\put(4,6){$01001$}
\put(25,10){$10001$}
\put(25,06){$00101$}
\put(35,20){$10100$}
\put(35,10){$10101$}
\put(25,30){$10010$}

\put(40,-5){\usebox{\gsix}}
\end{picture}
\caption{$\Gamma_5$ and $\Gamma_6$}
\end{figure}

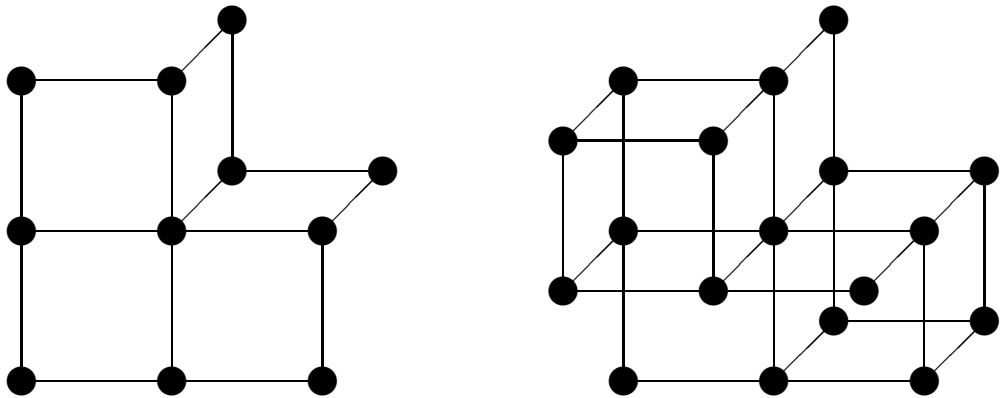
\begin{figure}[!p]
\centering
\setlength{\unitlength}{2mm}
\begin{picture}(80, 35)
 \newsavebox{\gggfour}
\savebox{\gggfour}
  (40,40)[bl]
  {
 
\put(10,10){\circle*{2}}
\put(10,20){\circle*{2}}
\put(10,30){\circle*{2}}
\put(20,10){\circle*{2}}
\put(20,20){\circle*{2}}
\put(20,30){\circle*{2}}
\put(30,10){\circle*{2}}
\put(30,20){\circle*{2}}

\put(10,10){\line(1,0){10}}
\put(20,10){\line(1,0){10}}
\put(10,20){\line(1,0){10}}
\put(20,20){\line(1,0){10}}
\put(10,30){\line(1,0){10}}
\put(10,10){\line(0,1){10}}
\put(10,20){\line(0,1){10}}
\put(20,10){\line(0,1){10}}
\put(20,20){\line(0,1){10}}
\put(30,10){\line(0,1){10}}
  }
\newsavebox{\lfive}
\savebox{\lfive}
  (40,40)[bl]
  {
  
\put(0,0){\usebox{\gggfour}}

\put(24,24){\circle*{2}}
\put(24,34){\circle*{2}}
\put(34,24){\circle*{2}}
\put(24,24){\line(1,0){10}}
\put(24,24){\line(0,1){10}}

\put(20,20){\line(1,1){4}}
\put(20,30){\line(1,1){4}}
\put(30,20){\line(1,1){4}} 
 }
 \newsavebox{\lsix}
\savebox{\lsix}
  (40,40)[bl]
  {
 \newsavebox{\gggfive}
\savebox{\gggfive}
  (40,40)[bl]
  {
  
\put(0,0){\usebox{\gggfour}}

\put(24,14){\circle*{2}}
\put(24,24){\circle*{2}}
\put(24,34){\circle*{2}}
\put(34,14){\circle*{2}}
\put(34,24){\circle*{2}}
\put(24,14){\line(1,0){10}}
\put(24,24){\line(1,0){10}}
\put(24,14){\line(0,1){10}}
\put(24,24){\line(0,1){10}}
\put(34,14){\line(0,1){10}}

\put(20,10){\line(1,1){4}}
\put(20,20){\line(1,1){4}}
\put(20,30){\line(1,1){4}}
\put(30,10){\line(1,1){4}}
\put(30,20){\line(1,1){4}}
  }
 \put(0,0){\usebox{\gggfive}}

\put(6,16){\circle*{2}}
\put(6,26){\circle*{2}}
\put(16,16){\circle*{2}}
\put(16,26){\circle*{2}}
\put(26,16){\circle*{2}}
\put(6,16){\line(1,0){10}}
\put(16,16){\line(1,0){10}}
\put(6,26){\line(1,0){10}}
\put(6,16){\line(0,1){10}}
\put(16,16){\line(0,1){10}}

\put(6,16){\line(1,1){4}}
\put(6,26){\line(1,1){4}}
\put(16,16){\line(1,1){4}}
\put(16,26){\line(1,1){4}}
\put(26,16){\line(1,1){4}}
  }
  
\put(0,-5){\usebox{\lfive}}
\put(40,-5){\usebox{\lsix}}
\end{picture}
\caption{$\Lambda_5$ and $\Lambda_6$}
\end{figure}

A {\em Fibonacci string} of length $n$ is a binary string
$b_1b_2\ldots b_n$ with $b_ib_{i+1}=0$ for $1\leq i<n$. 
The {\em Fibonacci cube} $\Gamma_n$ ($n\geq 1$) is the subgraph 
of $Q_n$ induced by the Fibonacci strings of length $n$. 
For convenience we also consider the empty string and set $\Gamma_0 = K_1$. 
Call a Fibonacci string $b_1b_2\ldots b_n$ a 
{\em Lucas string} if $b_1b_n \neq 1$. Then the {\em Lucas cube} 
$\Lambda_n$ ($n\geq 1$) is the subgraph of $Q_n$ induced by 
the Lucas strings of length $n$. We also set $\Lambda_0=K_1$.\\
Since their introduction  $\Gamma_n$ and $\Lambda_n$ have been also studied for their graph theory properties and found other applications, for example in chemistry (see the survey \cite{Klavzarsurvey}). Recently different enumerative sequences of these graphs have been determined. Among them: number of vertices of a given degree\cite{KlavzarDegree}, number of vertices of a given eccentricity\cite{Castro}, number of pair of vertices at a given distance\cite{KlavzarWiener} or number of isometric subgraphs isomorphic so some $Q_k$\cite{KlavzarCube}. The counting polynomial of this last sequence is known as cubic polynomial and has very nice properties\cite{cubepolmed}.

We propose to study an other enumeration and characterization problem.
For a given interconnection topology it is important to characterize maximal hypercubes, for example from the point of view of embeddings.
So let us consider {\em maximal hypercubes of dimension $p$}, i.e. induced subgraphs $H$  of $\Gamma_n$ (respectively  $\Lambda_n$) that are isomorphic to  $Q_p$, and such that there exists no induced subgraph $H'$ of $\Gamma_n$ (respectively  $\Lambda_n$), $H\subset H'$, isomorphic to $Q_{p+1}$.

Let $f_{n,p}$ and $g_{n,p}$ be the  numbers of maximal hypercubes of dimension $p$ of $\Gamma_n$, respectively  $\Lambda_n$, and $C'(\Gamma_{n},x)=\sum_{p=0}^{\infty}{f_{n,p}x^p}$, respectively $C'(\Lambda_{n},x)\sum_{p=0}^{\infty}{g_{n,p}x^p}$, their counting polynomials.

By direct inspection, see figures 1 to 3, we obtain the first of them:
\begin{eqnarray*}
C'(\Gamma_{0},x) & = & 1\,\:\ \ \ \  \ \  \ \ \ \ C'(\Lambda_{0},x)  =  1\, \\ 
C'(\Gamma_{1},x) & = & x\,\  \  \ \  \ \  \ \ \ \ C'(\Lambda_{1},x)  =  1\, \\ 
C'(\Gamma_{2},x) & = & 2x\,\  \ \  \ \  \ \ \ \ C'(\Lambda_{2},x)  =  2x \\
C'(\Gamma_{3},x) & = & x^2+x\,\  \ \ \     C'(\Lambda_{3},x)  =  3x\, \\ 
C'(\Gamma_{4},x) & = & 3x^2\,\  \   \ \    \ \ \ \ C'(\Lambda_{4},x)  =  2x^2\, \\ 
C'(\Gamma_{5},x) & = & x^3+3x^2\,\  C'(\Lambda_{5},x)  =  5x^2 \\ 
C'(\Gamma_{6},x) & = & 4x^3+x^2\,\  C'(\Lambda_{6},x)  =  2x^3+3x^2 \\ 
\end{eqnarray*}

The intersection graph of maximal hypercubes (also called cube graph) in a graph have been studied by various authors, for example in the context of median graphs\cite{Bresar}. Hypercubes playing a role similar to cliques in clique graph.
Nice result have been obtained on cube graph of median graphs, and it is thus of interest, from the graph theory point of view, to characterize maximal hypercubes in families of graphs and thus obtain non trivial examples of such graphs.
We will first  characterize maximal induced hypercubes in  $\Gamma_n$ and $\Lambda_n$ and then deduce the number of maximal $p$-dimensional hypercubes in these graphs.
\section{Main results}

For any vertex $x=x_1\dots x_n$ of $Q_n$ and any $i\in\{1,\dots,n\}$ let $x+\epsilon_i$ be the vertex of  $Q_n$ defined by $(x+\epsilon_i)_i= 1-x_i$ and $(x+\epsilon_i)_j= x_j$ for $j\neq i$.

Let $H$ be an induced subgraph of $Q_n$ isomorphic to some $Q_k$. The \emph{support} of $H$ is the subset set of $\{1\dots n\}$ defined by $Sup(H)=\{i/\: \exists\: x,y\in V(H)$ with $x_i\neq y_i\}$. Let $i\notin Sup(H)$, we will denote by $H\widetilde{+}\epsilon_i$ the subgraph induced by $V(H)\cup\{x+\epsilon_i/x\in V(H)\}$. Note that $H\widetilde{+}\epsilon_i$ is isomorphic to $Q_{k+1}$.

The following result is well known\cite{Klavzarnbhyper}.

\begin{proposition}\label{pro:bt}
In every induced subgraph $H$ of $Q_n$ isomorphic to $Q_k$ there exists a unique vertex of minimal weight, \emph{the bottom vertex} $b(H)$. There exists also a unique vertex of maximal weight, the \emph{top vertex} $t(H)$. Furthermore $b(H)$ and $t(H)$ are at distance $k$ and characterize $H$ among the subgraphs of $Q_n$ isomorphic to $Q_k$. 
\end{proposition}
We can precise this result.
A basic property of hypercubes is that if $x, x+\epsilon_i, x+\epsilon_j$ are vertices of $H$ then $x+\epsilon_i+\epsilon_j$ must be a vertex of $H$. By connectivity we deduce that if $x, x+\epsilon_i$ and $y$ are vertices of $H$ then  $y+\epsilon_i$ must be also a vertex of $H$. We have
thus by induction on $k$:

\begin{proposition}\label{pro:btprec}
if $H$ is an induced subgraph of  $Q_n$ isomorphic to $Q_k$ then
\begin{wlist}
  \item[(i)]$|Sup(H)|=k$
	\item[(ii)]$\text{If } i\notin Sup(H) \text{ then }\forall x\in V(H) \ \;x_i=b(H)_i=t(H)_i$
  \item[(iii)]$\text{If } i\in Sup(H) \text{ then }b(H)_i=0 \text{ and } t(H)_i=1 $
  \item[(iv)]$V(H)= \{x=x_1\dots x_n/ \;\forall i\notin Sup(H)\ x_i=b(H)_i\}.$

 \end{wlist} 
\end{proposition}
If $H$ is an induced subgraph of $\Gamma_n$, or $\Lambda_n$, then, as a set of strings of length $n$, it defines also an induced subgraph of $Q_n$; thus Propositions \ref{pro:bt} and \ref{pro:btprec} are still true for induced subgraphs of Fibonacci or Lucas cubes.
 
A Fibonacci string can be view as blocks of $0$'s separated by isolated $1$'s, or as isolated $0$'s possibly separated by isolated $1$'s.
These two points of view give the two following decompositions of the vertices of $\Gamma_n$. 
\begin{proposition}\label{pro:dec0}
 Any vertex of weight $w$ from $\Gamma_n$ can be uniquely decomposed as $0^{l_0}10^{l_1}\dots 10^{l_i}\dots10^{l_{p}}$ where 
 $p=w$; $\sum_{i=0}^p{l_i}=n-w$;  $l_0,l_p \geq 0$ and $l_1,\dots,l_{p-1}\geq1$.
 \end{proposition}
\begin{proposition}\label{pro:dec1}
 Any vertex of weight $w$ from $\Gamma_n$ can be uniquely decomposed as $1^{k_0}01^{k_1}\dots 01^{k_i}\dots01^{k_{q}}$ where 
$q=n-w$; $\sum_{i=0}^q{k_i}=w$  and $k_0,\dots,k_{q}\leq1$.
 \end{proposition}
\begin{proof}
A vertex from $\Gamma_n$, $n\geq2$ being the concatenation of a string of $V(\Gamma_{n-1})$ with $0$ or a string of $V(\Gamma_{n-2})$ with $01$, both properties are easily proved by induction on $n$.

\end{proof}
\qed

Using the the second decomposition, the vertices of weight $w$ from $\Gamma_n$ are thus obtained by choosing, in $\{0,1,\dots,q\}$, the $w$ values of $i$ such that $k_{i}=1$ in . We have then the classical result:
\begin{proposition}\label{pro:nbw}
For any $w\leq n$ the number of vertices of weight $w$ in $\Gamma_n$ is $\binom{n-w+1}{w}$.
\end{proposition}

Considering the constraint on the extremities of a Lucas string we obtain the two following decompositions of the vertices of $\Lambda_n$.

\begin{proposition}\label{pro:declucas0}
 Any vertex of weight $w$ in $\Lambda_n$ can be uniquely decomposed as $0^{l_0}10^{l_1}\dots 10^{l_i}\dots 10^{l_{p}}$ where 
 $p=w$, $\sum_{i=0}^p{l_i}=n-w$,  $l_0,l_p\geq0$, $l_0+l_p\geq1$ and $l_1,\dots,l_{p-1}\geq1$.
 \end{proposition}
 
\begin{proposition}\label{pro:declucas1}
 Any vertex of weight $w$ in $\Lambda_n$ can be uniquely decomposed as $1^{k_0}01^{k_1}\dots 01^{k_i}\dots01^{k_{q}}$ where $q=n-w$; $\sum_{i=0}^q{k_i}=w$; $k_0+k_q\leq1$ and $k_0,\dots,k_{q}\leq1$.
 \end{proposition}
 
From Propositions \ref{pro:btprec} and \ref{pro:dec0} it is possible to characterize the bottom and top vertices of maximal hypercubes in $\Gamma_n$.
\begin{lemma}
 If $H$ is a maximal hypercube of dimension $p$ in $\Gamma_n$ then $b(H)=0^n$ and $t(H)=0^{l_0}10^{l_1}\dots 10^{l_i}\dots10^{l_{p}}$ where $\sum_{i=0}^p{l_i}=n-p$; $0\leq l_0\leq 1$; $0\leq l_p\leq 1$  and $1\leq l_i\leq 2$ for $i=1,\dots,p-1$.
 Furthermore any such vertex is the top vertex of a unique maximal hypercube.
 \end{lemma}
 \begin{proof}
 Let $H$ be a maximal hypercube in $\Gamma_n$.
 Assume there exists an integer  $i$ such that  $b(H)_i = 1$. Then $i \notin sup(H)$ by Proposition \ref{pro:btprec}. Therefore, for any  $x \in V(H)$, $x_i= b(H)_i= 1$ thus $x+\epsilon_i \in V(\Gamma_n)$.
 Then $H\widetilde{+}\epsilon_i$ must be an induced subgraph of $\Gamma_n$, a contradiction with $H$  maximal.
    
 Consider now $t(H)=0^{l_0}10^{l_1}\dots 10^{l_i}\dots10^{l_{p}}$.
 
If $l_0\geq2$ then for any vertex $x$ of $H$ we have $x_0=x_1=0$, thus $x+\epsilon_0$ $\in V(\Gamma_n)$.  Therefore  $H\widetilde{+}\epsilon_i$ is an induced subgraph of $\Gamma_n$, a contradiction with $H$  maximal. The case  $l_p\geq2$ is  similar by symmetry.

Assume now $l_i\geq3$, for some $i\in\{1,\dots,p-1\}$. Let $j=i+\sum_{k=0}^{i-1}{l_k}$. We have thus $t(H)_j=1$ and $t(H)_{j+1}=t(H)_{j+2}=t(H)_{j+3}$=0. Then for any vertex $x$ of $H$ we have $x_{j+1}=x_{j+2}=x_{j+3}=0$, thus $x+\epsilon_{j+2}$ $\in V(\Gamma_n)$ and $H$ is not maximal, a contradiction.

Conversely consider a vertex $z=0^{l_0}10^{l_1}\dots 10^{l_i}\dots10^{l_{p}}$ where $\sum_{i=0}^p{l_i}=n-p$; $0\leq l_0\leq 1$; $0\leq l_p\leq 1$  and $1\leq l_i\leq 2$ for $i=1,\dots,p-1$. Then, by Propositions \ref{pro:bt} and \ref{pro:btprec}, $t(H)=z$ and $b(H)=0^n$ define a unique hypercube $H$ in $Q_n$ isomorphic to $Q_p$  and clearly all vertices of $H$ are Fibonacci strings. Notice that for any $i\notin Sup(H)$  $z+\epsilon_i$ is not a Fibonacci string thus $H$ is maximal.
\qed
 \end{proof}
 
With the same arguments we obtain for Lucas cube:
 \begin{proposition}\label{pro:lucas}
 If $H$ is a maximal hypercube of dimension $p\geq1$ in $\Lambda_n$ then $b(H)=0^n$ and $t(H)=0^{l_0}10^{l_1}\dots 10^{l_i}\dots10^{l_{p}}$ where $\sum_{i=0}^p{l_i}=n-p$; $0\leq l_0\leq 2$; $0\leq l_p\leq 2$; $1\leq l_0+l_p\leq 2$ and $1\leq l_i\leq 2$ for $i=1,\dots,p-1$.
 Furthermore any such vertex is the top vertex of a maximal hypercube.
 \end{proposition}
 
\begin{theorem}\label{th:fp}
 Let $0\leq p \leq n$ and $f_{n,p}$ be the number of maximal hypercubes of dimension $p$ in $\Gamma_n$ then:\\

\[f_{n,p}=\binom{p+1}{n-2p+1}\]\\

\end{theorem}
\begin{proof}
This is clearly true for $p=0$ so assume $p\geq1$. Since maximal hypercubes of $\Gamma_n$ are characterized by their top vertex, let us consider the set $T$ of strings which can be write $0^{l_0}10^{l_1}\dots 10^{l_i}\dots10^{l_{p}}$ where $\sum_{i=0}^p{l_i}=n-p$; $0\leq l_0\leq 1$; $0\leq l_p\leq 1$  and $1\leq l_i\leq 2$ for $i=1,\dots,p-1$. Let $l'_i =l_i-1$ for $i=1,\dots,p-1$; $l'_0 =l_0$; $l'_p =l_p$. We have thus a 1 to 1 mapping between $T$ and the set of strings $D=\{0^{l'_0}10^{l'_1}\dots 10^{l'_i}\dots10^{l'_{p}}\}$ where $\sum_{i=0}^p{l'_i}=n-2p+1$ any $l'_i\leq 1$ for $i=0,\dots,p$. This set is in bijection with  the set $E=\{1^{l'_0}01^{l'_1}\dots 01^{l'_i}\dots01^{l'_{p}}\}$. By Proposition \ref{pro:dec1}, $E$ is the set of Fibonacci strings of length $n-p+1$ and weight $n-2p+1$ and we obtain the expression of $f_{n,p}$ by Proposition \ref{pro:nbw}.
\qed
\end{proof}
\begin{corollary}
 The counting polynomial $C'(\Gamma_{n},x)=\sum_{p=0}^{\infty}{f_{n,p}x^p}$ of the number of maximal hypercubes of dimension $p$ in $\Gamma_n$ satisfies:
\begin{eqnarray*}
C'(\Gamma_{n},x)& = &x(C'(\Gamma_{n-2},x)+C'(\Gamma_{n-3},x))\ \ \ (n\geq3)\\
C'(\Gamma_{0},x)& = &1,\ C'(\Gamma_{1},x)=x,\ C'(\Gamma_{2},x)=2x\\
\end{eqnarray*}
The generating function of the sequence $\{C'(\Gamma_{n},x)\}$ is:
$$\sum_{n\geq0}{C'(\Gamma_{n},x)y^n}=\frac{1+xy(1+y)}{1-xy^2(1+y)} $$
\end{corollary}
\begin{proof}
By theorem \ref{th:fp} and Pascal identity we obtain $f_{n,p}=f_{n-2,p-1}+ f_{n-3,p-1}$ for $n\geq3$ and $p\geq1$. Notice that $f_{n,0}=0$ for $n\neq 0$. The recurrence relation for $C'(\Gamma_{n},x)$ follows.
Setting $f(x,y)=\sum_{n\geq0}{C'(\Gamma_{n},x)y^n}$ we deduce from the recurrence relation
$f(x,y)-1-xy-2xy^2=x(y^2(f(x,y)-1)+y^3f(x,y))$ thus the value of $f(x,y)$.
\qed
\end{proof}

\begin{theorem}\label{th:gp}
 Let $1\leq p \leq n$ and $g_{n,p}$ be the number of maximal hypercubes of dimension $p$ in $\Lambda_n$  then:\\
 
\[g_{n,p}=\frac{n}{p}\binom{p}{n-2p}.\]\\

\end{theorem}\begin{proof}
The proof is similar to the previous result with three cases according to the value of $l_{0}$:
 
By Proposition \ref{pro:lucas} the set $T$ of top vertices that begin with $1$ is the set of strings which can be write $10^{l_1}\dots 10^{l_i}\dots10^{l_{p}}$ where $\sum_{i=1}^p{l_i}=n-p$ and $1\leq l_i\leq 2$ for $i=1,\dots,p$.
Let $l'_i =l_i-1$ for $i=1,\dots,p$. We have thus a 1 to 1 mapping between $T$ and the set of strings $D=\{10^{l'_1}\dots 10^{l'_i}\dots10^{l'_{p}}$\} where $\sum_{i=1}^p{l'_i}=n-2p$, $0\leq l'_i\leq 1$ for $i=1,\dots,p$. Removing the first $1$, and by complement, this set is in bijection with  the set $E=\{1^{l'_1}\dots 01^{l'_i}\dots01^{l'_{p}}\}$. By Proposition \ref{pro:dec1}, $E$ is the set of Fibonacci strings of length $n-p-1$ and weight $n-2p$. Thus $|T|=\binom{p}{n-2p}$.

The set $U$ of top vertices that begin with $01$ is the set of strings which can be write $010^{l_1}\dots 10^{l_i}\dots10^{l_{p}}$ where $\sum_{i=1}^p{l_i}=n-p-1$; $1\leq l_i\leq 2$ for $i=1,\dots,p-1$ and $l_p\leq 1$.
Let $l'_i =l_i-1$ for $i=1,\dots,p-1$ and $l'_p =l_p$ . We have thus a 1 to 1 mapping between $U$ and the set of strings $F=\{010^{l'_1}\dots 10^{l'_i}\dots10^{l'_{p}}$\} where $\sum_{i=1}^p{l'_i}=n-2p$ and $l'_i\leq 1$ for $i=1,\dots,p$. Removing the first $01$, and by complement, this set is in bijection with  the set $G=\{1^{l'_1}\dots 01^{l'_i}\dots01^{l'_{p}}\}$. By Proposition \ref{pro:dec1}, $G$ is  the set of Fibonacci strings of length $n-p-1$ and weight $n-2p$. Thus $|U|=\binom{p}{n-2p}$.

The last set, $V$, of top vertices that begin with $001$, is the set of strings which can be write $0010^{l_1}\dots 10^{l_i}\dots0^{l_{p-1}}1$ where $\sum_{i=1}^{p-1}{l_i}=n-p-2$ and $1\leq l_i\leq 2$ for $i=1,\dots,p-1$.
Let $l'_i =l_i-1$ for $i=1,\dots,p-1$. We have thus a 1 to 1 mapping between $V$ and the set of strings $H=\{0010^{l'_1}\dots 10^{l'_i}\dots0^{l'_{p-1}}1\}$ where $\sum_{i=1}^{p-1}{l'_i}=n-2p-1$ and $l'_i\leq 1$ for $i=1,\dots,p-1$. Removing the first $001$ and the last $1$, this set, again by complement, is in bijection with the set $K=\{1^{l'_1}\dots 01^{l'_i}\dots01^{l'_{p-1}}\}$. The set $K$ is the set of Fibonacci strings of length $n-p-3$ and weight $n-2p-1$. Thus $|V|=\binom{p-1}{n-2p-1}$ and $g_{n,p}=2\binom{p}{n-2p}+\binom{p-1}{n-2p-1}=\frac{n}{p}\binom{p}{n-2p}$.
\qed
\end{proof}
\begin{corollary}
 The counting polynomial $C'(\Lambda_{n},x)=\sum_{p=0}^{\infty}{g_{n,p}x^p}$ of the number of maximal hypercubes of dimension $p$ in $\Lambda_n$ satisfies:
\begin{eqnarray*}
C'(\Lambda_{n},x)& = &x(C'(\Lambda_{n-2},x)+C'(\Lambda_{n-3},x))\ \ \ (n\geq5)\\
C'(\Lambda_{0},x)& = &1,\ C'(\Lambda_{1},x)=1,\ C'(\Lambda_{2},x)=2x,\ C'(\Lambda_{3},x)=3x,\ C'(\Lambda_{4},x)=2x^2\\
\end{eqnarray*}
The generating function of the sequence $\{C'(\Lambda_{n},x)\}$ is:
$$\sum_{n\geq0}{C'(\Lambda_{n},x)y^n}=\frac{1+y+xy^2+xy^3-xy^4}{1-xy^2(1+y)} $$
\end{corollary}
\begin{proof}
Assume $n\geq 5$. Here also by theorem \ref{th:gp} and Pascal identity we get $g_{n,p}=g_{n-2,p-1}+ g_{n-3,p-1}$ for $n\geq5$ and $p\geq2$. Notice that when $n\geq5$ this equality occurs also for $p=1$ and $g_{n,0}=0$. The recurrence relation for $C'(\Lambda_{n},x)$ follows and  $g(x,y)=\sum_{n\geq0}{C'(\Lambda_{n},x)y^n}$ satisfies
$g(x,y)-1-y-2xy^2-3xy^3-2x^2y4=x(y^2(g(x,y)-1-y-2xy^2)+y^3(g(x,y)-1-y))$.
\qed
\end{proof}

Notice that $f_{n,p}\neq0$  if and only if $\left\lceil \frac{n}{3} \right\rceil \leq p  \leq \left\lfloor \frac{n+1}{2} \right\rfloor$ and $g_{n,p}\neq0$  if and only if $\left\lceil \frac{n}{3} \right\rceil \leq p  \leq \left\lfloor \frac{n}{2} \right\rfloor$ (for $n\neq1$). Maximal induced hypercubes of maximum dimension are maximum induced hypercubes and we obtain again that cube polynomials of $\Gamma_n$, respectively  $\Lambda_n$, are  of degree $\left\lfloor \frac{n+1}{2} \right\rfloor$, respectively $\leq \left\lfloor \frac{n}{2} \right\rfloor$ \cite{KlavzarCube}.

\end{document}